\documentclass[11pt,reqno]{amsart}
\usepackage{latexsym}
\usepackage{amssymb,amsfonts,amsmath,mathrsfs}
\usepackage{epstopdf}
 \usepackage{hyperref, amscd, flafter,epsf}
 \usepackage{graphicx}
 \usepackage{amssymb}
 \usepackage{amsbsy}
 \usepackage{amsmath}

\newtheorem{theorem}{Theorem}
\newtheorem{lemma}{Lemma}

\newtheorem*{theoremA}{Theorem A}
 \numberwithin{equation}{section}

\newcommand{\be}{\begin{equation}}
\newcommand{\ee}{\end{equation}}
\newcommand{\s}{\sigma}
\newcommand{\g}{\gamma}
\newcommand{\z}{\zeta}
\newcommand{\e}{\epsilon}

\newcommand{\meas}{\mathop{\rm meas}}

\newcommand{\ab}{\big|}

\newcommand{\AB}{\bigg|}
\newcommand{\bOmega}{{\boldsymbol \Omega}}
\newcommand{\vr}{\varrho}
\newcommand{\D}{\Delta}  
\newtheorem*{theorem*}{Theorem}

\def\({\left(}
\def\){\right)}
\newcommand\zpz{\frac{\zeta'}{\zeta}}
\newcommand{\intl}{\int\limits}

\newcommand\bea{\begin{eqnarray}}
\newcommand\eea{\end{eqnarray}}
\newcommand\bi{\begin{itemize}}
\newcommand\ei{\end{itemize}}
\newcommand\ben{\begin{enumerate}}
\newcommand\een{\end{enumerate}}
\newcommand\bes{\begin{equation*}}
\newcommand\ees{\end{equation*}}

\title[On the distribution of the zeros of $\zeta'(s)$]
 {On the distribution of the zeros of the derivative \\
 of the Riemann zeta-function} 
\author{S. J. Lester}
\address{Department of Mathematics, University of Rochester, Rochester, NY 14627 USA}

\curraddr{School of Mathematical Sciences, Tel Aviv University, Tel Aviv 69978, Israel}

\thanks{The author was supported in part by the NSF grant DMS-1200582.}

\email{slester@post.tau.ac.il}

\begin{document}

\begin{abstract}
We establish an unconditional asymptotic formula
describing the horizontal distribution of the zeros of the derivative
of the Riemann zeta-function. For $\Re(s)=\s$ satisfying 
$(\log T)^{-1/3+\e} \leq (2\s-1) \leq (\log \log T)^{-2}$,
we show that the number of zeros of $\z'(s)$
with imaginary part between zero and $T$ and real part larger than $\sigma$ 
is asymptotic to $T/(2\pi(\s-1/2))$ as $T \rightarrow \infty$.
This agrees with a prediction from random matrix theory due to Mezzadri.
Hence, for $\s$ in this range the zeros of $\zeta'(s)$
are horizontally distributed like the zeros of the derivative of characteristic polynomials
of random unitary matrices are radially distributed.
\end{abstract}

\maketitle

\section{Introduction}

Let $\z(s)$ denote the Riemann zeta-function with $s=\s+it$ a complex variable. Also, let $\rho=\beta+i\g$
be a non-trivial zero of $\z(s)$ and $\rho'=\beta'+i\g'$ be a non-trivial zero of $\z'(s)$.
Throughout we let $T$ denote a sufficiently large parameter.

The distribution of the zeros of $\z'(s)$ is closely connected 
to the distribution of the zeros of $\z(s)$. 
For instance, it follows from the work of Speiser ~\cite{Sp} 
that the Riemann hypothesis
 is equivalent to $\z'(s)$ not vanishing in the strip $0<\s<1/2$. 
 Levinson and Montgomery ~\cite{LM} 
showed that
\[
\sum_{\substack{0< \g' \leq T \\ 0<\beta'<1/2 }}1=
\sum_{\substack{0< \g \leq T \\ 0<\beta<1/2 }}1+O(\log T);
\]
and, if there are no more than $T/2$ zeros of $\z(s)$ with
$0<\beta<1/2$ and $0<\g\leq T$ then there is a sequence $\{T_n\}_{n=1}^{\infty}$, with $T_n\rightarrow \infty$
as $n \rightarrow\infty$, 
such that the above equation holds with no error term. 
In ~\cite{Le} Levinson showed
that more than $1/3$ of the zeros of $\z(s)$ lie on the critical line $\s=1/2$. His
method relies upon bounding the number of zeros of $\z'(s)$ with $0<\beta'<1/2$
and $0<\g'\leq T$.

There also exists a fascinating relationship between the
zeros of $\z'(s)$ that lie close to the critical line $\s=1/2$
and small gaps between ordinates of zeros of $\z(s)$. This connection was observed by C. R. Guo ~\cite{G} and has been actively studied since (see ~\cite{So}, ~\cite{Z}, ~\cite{F},
~\cite{GY}, ~\cite{K}, ~\cite{FK}, and ~\cite{Ra}).

In this article we are interested in the horizontal distribution of the zeros of $\z'(s)$ in the half-plane $\s>1/2$.
C. R. Guo ~\cite{Gg} also studied this problem.
Let
\[
\chi_p(u,v;\s)=\int_0^1 e\bigg(u \log p \sum_{m=1}^{\infty}
\frac{\cos(2\pi m \theta)}{p^{m\s}}-v\log p 
\sum_{m=1}^{\infty}\frac{\sin(2\pi m \theta)}{p^{m\s}}\bigg) d\theta,
\]
where $e(x)=e^{2\pi i x}$. Next, write
 $\chi(u,v;\s)=\prod_p \chi_p(u,v;\s)$. Also, let $\chi_{\sigma}$
be the partial derivative of $\chi$ with respect to $\sigma$,
and denote the Fourier transform of $\chi_{\sigma}$
by $\widehat \chi_{\sigma}$. Guo proved that for
$ 1/2 +2(\log T)^{-1/600} < \s< 2$, and any $\epsilon>0$ that
\be \label{unweighted}
\begin{split}
\sum_{\substack{0< \g' \leq T \\  \beta'>\s}} 1
=&-\frac{T}{4\pi} \iint_{\mathbb{R}^2} \widehat \chi_{\s} (x, y;\s)  \log(x^2+y^2)  \, dx \, dy\\
&-\frac{\log 2}{2\pi}T
+O(T(\log T)^{-\epsilon}).
\end{split}
\ee

Random matrix analogues of this problem have also been studied.
To state these results, let $U(N)$ be the group of $N \times N$ unitary
matrices. For $U \in U(N)$ write $Z(U,z)$ for the characteristic polynomial
of $U$. Also,
let $\lambda'$ denote a zero of $\frac{d}{dz}Z(U,z)$.
Mezzadri \cite{Me}
proved that as $N \rightarrow \infty$
\be \label{Mezz form}
\mathbb E_{U(N)} \bigg( \frac{1}{N-1} \sum_{ |\lambda'|<1-x/N} 1 \bigg)
\sim \frac1x \qquad \qquad x\rightarrow \infty, 
\quad x=o(N).
\ee
Here the expected
value is with respect to Haar measure on $U(N)$. Additionally, Mezzadri
argued that
\[
\lim_{N \rightarrow \infty}
\mathbb E_{U(N)} \bigg( \frac{1}{N-1} \sum_{ |\lambda'|  \geq 1- x/N } 1 \bigg)
\sim \frac{8}{9 \pi} x^{3/2}
\]
as $x \rightarrow 0^+$ (see also Due\~nez \emph{et al.} ~\cite{DF}).

Analogues for the zeros of 
$\z'(s)$ of these formulas lead to
the following conjectures:
\begin{itemize}
\item[a)] for $\s=1/2+\psi(T)/\log T$, where $\psi(T)=o(\log T)$ and $\psi(T) \rightarrow \infty$
as $T \rightarrow \infty$, we have
\be \label{mezz pred 1}
\sum_{\substack{0 <\g'\leq T \\ \beta'>\s}}1
=\frac{T}{2\pi(\s-\tfrac12)}(1+o(1)) \qquad (T \rightarrow \infty);
\ee
\item[b)] for $\nu\rightarrow 0^+$, we have
\[
\lim_{T \rightarrow \infty}
\frac{1}{N'(T)}
\sum_{\substack{0< \g' \leq T \\ \beta'<\tfrac12+\nu/\log T}}1\sim\frac{8}{9 \pi} \nu^{3/2},
\]
\end{itemize}
where
\be \label{sum zetap zeros}
N'(T)=\sum_{0< \g' \leq T} 1=\frac{T}{2\pi} \log \frac{T}{2\pi }
-(1+\log 2) \frac{T}{2\pi}+O(\log T)
\ee
(see Berndt ~\cite{Be}). 
Under the assumption of the Riemann hypothesis
and Montgomery's pair correlation conjecture 
M. Radziwi\l\l
~\cite{Ra} has shown that
\[
\nu^{3/2+\varepsilon} \ll
\liminf_{T \rightarrow \infty}
\frac{1}{N'(T)}
\sum_{\substack{0< \g' \leq T \\ \beta'<\tfrac12+\nu/\log T}}1\ll \nu^{3/2-\varepsilon},
\]
for any $\varepsilon>0$ as $\nu \rightarrow 0^+$.

In this article we derive an explicit, unconditional formula for the sum on the left-hand side
 of \eqref{unweighted}. Moreover, we show that \eqref{mezz pred 1} holds
for $\s\in(1/2+(\log \log T)^2/(\log T)^{1/3}, 1/2+1/(\log \log T)^2)$. Therefore, by setting $\sigma=1/2+x/\log T$ in this formula and 
taking $N=\lfloor \log T \rfloor$ in \eqref{Mezz form} we see that
\[
\frac{1}{N'(T)} \sum_{\substack{0 <\g'\leq T \\ \beta'>\frac12+x/\log T}}1
\sim \mathbb E_{U(N)} \bigg( \frac{1}{N-1} \sum_{ |\lambda'|<1-x/N} 1 \bigg) \qquad \qquad (T \rightarrow \infty)
\]
for $(\log \log T)^2 (\log T)^{2/3} \leq x \leq \log T/(\log \log T)^2$.
Thus, in this range
the zeros of $\z'(s)$
are  horizontally distributed 
like the zeros of the derivative of characteristic polynomials
of random unitary matrices are radially distributed. 
\section{Main results}

\begin{theorem} \label{main est} Let $\psi(T)=(2\s-1)\log T$ and $\mathcal L_2=\log \log T$.
For \\
 $  \mathcal L_2/(\log T)^{1/3} \leq (2\s-1) \leq 1/(10 \mathcal L_2)$ we have
\bes
\begin{split}
\int_0^T \! \log \AB \zpz\(\s+it\)  \AB \, dt=T \log \Big(\frac{1}{2\s-1}\Big)-\frac{\g}{2} T
+O\(T E(\psi(T)) \),
\end{split}
\ees
where $\g=0.57721566...$ is Euler's constant and
\[
E(\psi(T)) =
\begin{cases}
\displaystyle  \frac{\mathcal L_2^3 (\log T)^{3/2}}{\psi(T)^{9/4}} & 
\text{if   } (\log T)^{2/3} \mathcal L_2 \leq \psi(T) 
\leq \mathcal L_2^{4/17}(\log T)^{14/17}, \\
\displaystyle  \frac{\mathcal L_2^{2} \psi(T)^2}{\log^2 T} & 
\text{if   } \mathcal L_2^{4/17} (\log T)^{14/17}  \leq \psi(T) 
\leq \log T/(10 \mathcal L_2).
\end{cases}
\]
\end{theorem}  

Note that $E(\psi(T))=E((2\sigma-1)\log T) =o(1)$ as $T \rightarrow \infty$ for
$ (\log \log T)^2/(\log T)^{1/3} \leq (2\s-1) \leq 1/(\log \log T)^2$.

Using Littlewood's Lemma (see Section \ref{section 4})
and applying the previous result
we then prove
\begin{theorem} \label{thm 1}
Let $\psi(T)=(2\s-1)\log T$ and $\mathcal L_2=\log \log T$.
 For \\
 $ \mathcal L_2/(\log T)^{1/3} \leq (2\s-1) \leq 1/(10\mathcal L_2)$ we have
\be \label{first result}
\begin{split}
 2 \pi \sum_{\substack{  0< \g' \leq T \\  \beta'>\s}} (\beta'-\s)=T \log \Big( \frac{1}{2\s-1}\Big)+\Big(\log \frac{2^{\s}}{\log 2}-\frac{\g}{2}\Big) T+O(TE(\psi(T))),
\end{split}
\ee
where $\g=0.57721566...$ is Euler's constant and $E(\psi(T))$
is as in Theorem \ref{main est}.
\end{theorem}

To calculate the horizontal distribution of 
the zeros of $\zeta'(s)$
we remove the weight $(\beta'-\s)$ from 
the sum on the left-hand side of \eqref{first result}. To do this,
observe that
\[
 \sum_{\substack{ 0< \g' \leq T \\  \beta'>\s}}( \beta'-\s)=\int_{\s}^{\infty} \sum_{ \substack{0< \g' \leq T \\  \beta'>u}} 1 \, du.
\]
We essentially differentiate
\eqref{first result} (the precise argument will be given
in Section \ref{section 5}) and prove
\begin{theorem} \label{thm 2.1}
Let $\psi(T)=(2\s-1)\log T$ and $\mathcal L_2=\log \log T$.
For \\
$  \mathcal L_2^2/(\log T)^{1/3} \leq (2\s-1) \leq 1/(20 \mathcal L_2)$ we have
\be \label{dist of zeros}
\sum_{\substack{ 0< \g' \leq T \\  \beta'>\s}} 1=\frac{T}{2\pi(\s-\tfrac12)}\Big(1+O\Big(\sqrt{ E (\psi(T))}\Big)\Big),
\ee
where $E(\psi(T))$ is as in Theorem \ref{main est}.
\end{theorem}

\section{The Proof of Theorem \ref{main est}}
\subsection{The approach to the proof of Theorem \ref{main est}}
For $2\s-1 \geq 1/\log T$ let
\be\notag
V=\frac12 \sum_{n=2}^{\infty} \frac{\Lambda^2(n)}{n^{2\s}}=\frac{1}{2(2\s-1)^2}+O(1).
\ee
In particular,
\be\label{V bd}
V\ll \log^2 T.
\ee

A main component in the proof of Theorem \ref{main est}
is the following
result on the distribution of $\z'/\z(\s+it)$
that was proved by the author in ~\cite{L} (see also ~\cite{lester arxiv}) .
We state this as
\begin{theoremA} \label{theorem disk}
Let $\psi(T)=(2\s-1) \log T$ and
\[
\bOmega= e^{-10} \min\big(V^{3/2}, (\psi(T)/ \log \psi(T))^{1/2}\big).  
\]
Suppose that $\psi(T)\to \infty$ with $T$,  $\psi(T)=o(\log T)$,
 and that  $r$ is a real number such that $r \bOmega \geq 1$.
Then we have
\be \label{main measure}
\meas \bigg\{ t \in (0, T) : \bigg|\frac{\z'}{\z}(\s+it)  \bigg| \leq \sqrt V r  \bigg \}
=  T(1-e^{-r^2/2}) +O\(T\(\frac{r^2+r}{\bOmega}\)\),
\ee
where $\meas$ refers to Lebesgue measure.
If, in addition, we let 
\be \label{e def}
\varepsilon= 
\max\big( e^{\s/(1-2\s)}/(2\s-1),e^{10}(\log \psi(T)/\psi(T) )^{1/2}\big)
\ee
then we have
for $r\geq \varepsilon$ that  
\be \label{small measure}
\meas\bigg\{ t \in (0, T) : \bigg|\frac{\z'}{\z}(\s+it) V^{-1/2} \bigg| \leq r \bigg\} 
\ll T r^2.
\ee
 \end{theoremA}

Our goal is to estimate
\[
\int_0^T \! \log \AB \zpz\(\s+it\) V^{-1/2} \AB \, dt.
\]
We begin by writing
\bes
\int_0^T \log \AB \zpz(\s+it) V^{-1/2} \AB \, dt=I_0+ I_M+I_{\infty},
\ees
where $I_0$ is the portion of the integral over the set where $|\z'/\z(\s+it) V^{-1/2}|$ is less than 
$\varepsilon$ ($\varepsilon$ is defined in \eqref{e def}),
$I_{\infty}$ is the portion of the integral where $|\z'/\z(\s+it) V^{-1/2} |$ is greater than 
\be \label{def N}
\mathcal N=cV^{-1/2} \log^2 T  
\ee
with $c>0$ to be specified below,
and $I_M$ is the rest.

In the next section we will estimate $I_{\infty}$. To do this, first note 
that for $10< t< T$, unless there is a zero of $\z(s)$ within a distance of $1/\log T$ 
of $s$, we have $|\z'/\z(s)| \leq c_1  \log^2 T$
for some absolute constant $c_1>0$.
This   follows from the well-known formulas
\be \label{landau's formula}
\zpz(s)=\sum_{|t-\g| \leq 1} \frac{1}{s-\rho}+O\big(\log t\big)
\ee
and $N(t+1)-N(t) \ll \log t$, where $N(t)$ is the number
of zeros of $\z(s)$ with $0<\gamma\leq t$ (see Titchmarsh ~\cite{T}, Chapter IX).  In the definition
of $\mathcal N$ in \eqref{def N} we now choose $c=c_1$. Using Selberg's~\cite{Se} zero density estimate
\be\label{zero density}
N(\s, T) =\sum_{\substack{0< \g \leq T \\  \beta>\s}} 1 \ll T^{1-(\s-1/2)/4}\log T,
\ee
we will then show that the measure of the set where $|\z'/\z(s)|V^{-1/2}$ is greater than $\mathcal N$ is quite small.
Combining this estimate with the crude bound
\[
\int_0^T \bigg(\log \AB\zpz(\s+it)\AB \bigg)^2 \, dt \ll T \log^2 T ,
\]
which we prove later, we then obtain a bound on $I_{\infty}$.

We next consider $I_{0}$.
The estimation of $I_0$ requires a more complicated argument.
There are essentially two main steps.
First, we use a few lemmas from complex analysis  to obtain a
pointwise lower bound on $\log |\z'/\z(s)|$ in terms of a sum
over the zeros of $\z'/\z(s)$ and the maximum value of
$|\z'/\z(s)|$ in a disk that does not contain a zero of $\z(s)$. 
We then use \eqref{small measure}
to obtain a bound on the set of $t$ between zero and $T$ for which $\z'/\z(\s+it)$ is small. 

Finally, to evaluate $I_M$  let
\[
\Psi_T(r)= \frac1T
\meas \bigg\{ t \in (0, T) : 
\bigg|\frac{\z'}{\z}(\s+it) V^{-1/2} \bigg| \leq  r  \bigg \}.
\]
Using the
 change of variable formula (see ~\cite{Dur} Theorem 1.6.9) we
observe that
\[
I_M= \int_{[\e,\mathcal N]} \! \log r \, \, d\Psi_T(r).
\]
We then use \eqref{main measure} to evaluate the integral on the right-hand side.

\subsection{Preliminary Lemmas} \label{complex anal}

Let $D_r(z)$ 
be the closed disk of radius $r$ centered at $z$.
We now consider functions $F(s)$ such that:
 \bi
 \item[i)] 
   $F(s)$ is analytic in $D_R(\s_0+it)$, where $\s_0 \in \mathbb{R}$ and $R>0$;
 \item[ii)]  $C \leq |F(s)| \leq 1/C $ in $ D_r(\s_0+it)$, where $r < R$ and $0< C < 1$. 
 \ei
Next, let $\vr$ denote a zero of $F(s)$ and let
 \[
 N_t( R)=\sum_{\vr \in D_{ R}(\s_0+it) } 1, 
 \]
 where each zero is counted according to its multiplicity. Also, set
 $s_0=\s_0+it$,
 and let $M_t( R)=\max_{|z-s_0| \leq  R } |F(z)|+2$.
\begin{lemma} Suppose that $F(s)$ satisfies conditions $ i)$ and $ ii)$ above
and that $R$ is bounded above by $A$ for some $A>1$.
Then for $0< \delta < R$ we have
\be\label{zero bd}
N_t(R-\delta) \ll \delta^{-1} \log M_t(R),
\ee
where the implied constant depends only on $C$, $\s_0$, and $A$.
\end{lemma}

\begin{proof}
By Jensen's formula (see Titchmarsh ~\cite{Ti})
\[
\log|F(s_0)|=-\sum_{\vr \in D_R(\s_0+it)} \log \( \frac{R}{| \vr|}\)+\frac{1}{2\pi} \int_0^{2\pi} \log |F(R e^{i\theta}+s_0)| \, d\theta.
\]
By condition ii), we have $\log |F(s_0)| \geq \log C$. Also, 
\[
\frac{1}{2\pi} \int_0^{2\pi} \log |F(Re^{i\theta}+s_0)| \, d\theta \leq \log M_t(R).
\]
Finally, note that
\[
\sum_{\vr \in D_R(\s_0+it)} \log \( \frac{R}{| \vr|}\)=\int_0^R N_t(x) \, \frac{dx}{x} \geq \int_{R-\delta}^R N_t(x) \, \frac{dx}{x} 
\geq  N_t(R-\delta)  \int_{R-\delta}^R\frac{dx}{x}. 
\]
As $R \leq A$, the right-hand side of the inequality is $\geq (\delta/A) N_t(R-\delta)$.
Using these bounds in Jensen's formula, we deduce that
\[
\delta N_t(R-\delta) \leq A \log (M_t(R)/C).
\]
\end{proof}

\begin{lemma} \label{decent bd}
Suppose that $F(s)$ satisfies conditions $i)$ and $ii)$   above
and that $R \leq A$, for some $A>1$. Let $0<\delta <\min(1,R)/2$ and
$F_1(s)=F(s) \prod_{|s_0-\vr| \leq R-\delta} (s-\vr)^{-1}$. Then
for $z \in D_{R-2\delta}(\s_0+it)$ we have
\[
|\log |F_1(z)|| \ll  \frac{\log (1/\delta)}{\delta^2}   \,  \log M_t(R),
\]
where the implied constant depends only on $C$, $\s_0$, $A$, and $r$.
\end{lemma}
\begin{proof}
Since $F(s)$ satisfies condition ii) above we know that $|F(s_0)| \geq C$. Therefore,
on the boundary of $D_R(\s_0+it)$ (and therefore inside $D_R(\s_0+it)$ as well),
we have that
\be\label{est F1 2}
\begin{split}
\bigg|\frac{F_1(s)}{F_1(s_0)}\bigg|=\bigg|\frac{F(s)}{F(s_0)}  
\prod_{|s_0- \vr| \leq R-\delta}  \frac{s_0- \vr}{s- \vr} \bigg| 
& \leq 1/C \max_{|z-s_0|=R} (|F(z)|+2)  \(\frac{R-\delta}{\delta}\)^{N_t(R-\delta)} \\
&  \leq 1/C \max_{|z-s_0|=R} (|F(z)|+2)  \(\frac{R}{2\delta}\)^{N_t(R-\delta)}:=K,
 \end{split}
\ee
as $R-\delta<R/2$.

We now define
\[
G_1(z)=F_1(z+\s_0+it) \qquad \mbox{and} \qquad H_1(z)= \log (G_1(z)/G_1(0)),
\] 
where the branch of the logarithm is chosen so that $\log 1=0$ (so in particular $H_1(0)=0$).
Since $F_1(s)$ is non-vanishing in $D_{R-\delta}(\s_0+it)$, $H_1(z)$ is analytic for $|z|\leq R-\delta$. We also see 
by \eqref{est F1 2} that for $|z| \leq R$,
\[
\Re H_1(z) \leq \log K.
\]
Thus, by the Borel-Caratheodory theorem (see Titchmarsh ~\cite{Ti}, we take our
radii to be $R-\delta$ and $R-2\delta$),
we have for $|z| \leq R - 2 \delta$ that
\[
|H_1(z)|\leq \frac{2 (R-2\delta)}{\delta} \log K.
\]
Thus, for
$|z-s_0| \leq R-2\delta$ we have that
\bes
 |\log (F_1(z)/F_1(s_0)) | \leq \frac{2 (R-2\delta)}{\delta} \log K.
\ees
The above inequality implies that
\be \label{est F1 3}
|\log |F_1(z)||\leq \frac{2 (R-2\delta)}{\delta} \log K+|\log |F_1(s_0)| |,
\ee
for $|z-s_0| \leq R-2\delta$.

Note that $F_1(s_0)=F(s_0) \prod_{|s_0-\vr| \leq R-\delta} (s_0-\vr)^{-1}$,
and that since $F(s)$ satisfies condition ii) above, we have
 \bes 
C (R-\delta)^{-N_t(R-\delta)} \leq \bigg|F (s_0) \prod_{|s_0- \vr| \leq R-\delta} (s_0- \vr)^{-1}\bigg| \leq r^{-N_t(R-\delta)} 1/C.
\ees
So by this and \eqref{zero bd} we obtain
\[
|\log |F_1(s_0)|| \ll N_t(R-\delta) +1\ll  \delta^{-1}  \log M_t(R). 
\]
Also, by \eqref{zero bd} 
 \[
 \log K = \log (M_t( R)/C)+N_t(R-\delta) \log (R/(2\delta)) \ll \frac{\log  (1/\delta)}{\delta}   \,  \log M_t(R).
 \]
Using these estimates in \eqref{est F1 3}, we obtain the result.
\end{proof}

 \subsection{Estimation of $I_{\infty}$}

 In this section we show that the possible poles
 of $\z'/\z(s)$ will not contribute much to 
 \[
 \int_0^T \log \bigg|\frac{\z'}{\z}(\s+it) \bigg| \, dt.
 \]
 Our initial goal is to obtain a crude bound on
\bes
\int_0^T \bigg(\log \bigg|\zpz(\s+it)\bigg|\bigg)^2 \, dt. 
\ees
We define $\log \zeta(s)$ and $\log \zeta'(s)$ 
in the standard way.
\begin{lemma} \label{log zeta prime lem}
Suppose $t \neq \gamma'$.
There exists a positive integer $N_0$ such that uniformly for $-1 \leq \s \leq 4$ 
and $t \geq N_0$, we have
\[
\log \z'(s)=\sum_{|\g'-t|\leq 1}\log(s-\rho')+O(\log t),
\]
where $-\pi < \Im (s-\rho') \leq \pi$.
\end{lemma}
\begin{proof}
The analogue of this formula for $\zeta(s)$
is well-known and is proved in Chapter IX of ~\cite{Ti}.
The lemma follows from a similar argument.
\end{proof}
\begin{lemma} \label{simple bd}
Uniformly for $ -1 \leq \s \leq 2$ we have
\[
\int_0^T \bigg(\log \AB \zpz(\s+it) \AB\bigg)^2 \, dt \ll T \log^2 T.
\]
\end{lemma}
\begin{proof}
By Lemma \ref{log zeta prime lem} 
we have uniformly  for $-1 \leq \s \leq 4$ that
\be \label{crude bd 1}
\begin{split}
\int_0^T \log ^2 | \z'(\s+it)| \, dt \ll &
\int_0^{N_0} \log ^2 | \z'(\s+it)| \, dt\\
&+\int_{N_0}^T \bigg(\sum_{|t-\g'| \leq 1} \log | \s+it-\rho' | \bigg)^2 \, dt 
+O(T \log ^2 T).
\end{split}
\ee
It follows from \eqref{sum zetap zeros} that $N'(t+1)-N'(t) \ll \log t$, for $t$ sufficiently large.
By this and Cauchy's inequality
\be \label{crude bd 2}
\int_{N_0}^T \bigg(\sum_{|t-\g '| \leq 1} \log | \s+it-\rho' |\bigg)^2 \, dt \ll \log T \int_{N_0}^T \sum_{|t-\g '| \leq 1} \log^2 | \s+it-\rho' | \, dt.
\ee

We now partition the interval $(N_0, \lceil T \rceil)$ 
into subintervals, each of length one, and interchange the sum 
and integral to see that
\bes
\begin{split}
\int_{N_0}^T\; \sum_{|t-\g'| \leq 1} \log^2 | \s+it-\rho' | \, dt 
\ll & \sum_{n=N_0}^{\lceil T \rceil} \;\int_n^{n+1}\; \sum_{|t-\g'| \leq 1} \log ^2 | \s+it-\rho' | \,dt \\
\ll & \sum_{n=N_0}^{\lceil T \rceil}\; \sum_{n-1 \leq \g' \leq n+2} \; \int_{\g'-1}^{\g'+1} \log^2 | \s+it-\rho' | \, dt.
\end{split}
\ees
Note that
\[
\int_{\g'-1}^{\g'+1} \log^2 | \s+it-\rho' | \,dt \ll 1.
\]
Hence, the above is 
\[
\ll \sum_{n=N_0}^{\lceil T \rceil} \sum_{n-1 \leq \g' \leq n+2}  1 \ll \sum_{n=N_0}^{\lceil T \rceil} \log (n+2) \ll T \log T.
\]
From this, \eqref{crude bd 1}, and \eqref{crude bd 2}, along with the simple observation that
$\int_0^{N_0} \log ^2 | \z'(\s+it)| dt \ll 1$, we have that
\[
\int_0^T \log^2 | \z'(\s+it) | \, dt \ll T \log ^2 T.
\]

A  similar argument
shows that, uniformly for $-1 \leq \s \leq 2$,
\[
\int_0^T \log^2 |\z(\s+it)| \, dt \ll T \log^2 T.
\]
The lemma now follows by noting that
\[
\int_0^T \bigg(\log \bigg| \zpz(\s+it) \bigg|\bigg)^2 \, dt \ll \int_0^T \log^2 |\z'(\s+it)| \, dt+\int_0^T \log^2 |\z(\s+it)| \, dt.
\]
\end{proof}

We can now prove 
 \begin{lemma} \label{I infty bd}
 Let $1/2+50 \log \log T/ \log T \leq \s \leq 2$. Then we have
 \bes
 I_{\infty} \ll \frac{T}{\log^2 T}.
 \ees
 \end{lemma}
 \begin{proof}
 Let $S_{\infty}$ be the subset of $t$ in
 $(0, T)$ for which $|\z'/\z(\s+it)| V^{-1/2} > \mathcal N$ (where $\mathcal N$
 is defined in \eqref{def N}). By \eqref{V bd} we have
 \[
 I_{\infty}=\intl_{S_{\infty}} \log \AB \zpz(\s+it) V^{-1/2} \AB \, dt
 =\intl_{S_{\infty}} \log \AB \zpz(\s+it) \AB \, dt+O\Big(\meas(S_{\infty})\log \log T \Big).
 \]
By Lemma \ref{simple bd} and Cauchy's inequality
\be \label{bo bound}
\intl_{S_{\infty}} \log  \bigg|\zpz(\s+it)\bigg| \, dt \ll \(\meas \(S_{\infty}\)\)^{1/2}T^{1/2} \log T.
\ee

We now let $\s'=1/2+49\log \log T/ \log T$, so that $\s -1/\log T> \s'$.
To estimate the size of $S_{\infty}$ consider
\be \label{A inf def}
A_{\infty}=\bigcup\limits_{\substack{ 0< \g \leq T+1 \\ \beta > \s' }} \Big\{ t \in (0, T) : |t-\g| \leq 1/\log T  \Big\}.
\ee
By definition, if $t \in S_{\infty}$ there is a zero 
of $\z(s)$ satisfying $|\rho-(\s+it)| \leq 1/\log T$. Thus,
$|\beta-\s| \leq 1/ \log T$  and $|t-\g| \leq 1/\log T$.  In particular, $\beta\geq\s-1/\log T>\s'$; 
it follows that 
$t \in A_{\infty}$. Thus,
\be \label{A inf bd}
\meas \(S_{\infty}\) \leq \meas \( A_{\infty}\) 
\ll  \sum_{\substack{ 0<\g \leq T+1 \\ \beta> \s'}} (1/\log T) \ll T^{1-(\s'-1/2)/4} \ll \frac{T}{\log^{12} T}
\ee
by \eqref{zero density}. 
Using this in \eqref{bo bound}
we complete the proof.
\end{proof}

 \subsection{Estimation of $I_0$}
Let $\mathbf R=\s_0-(\s+1/2)/2$ and let $\s_0$ be a fixed, suitably large 
constant ($\s_0=5$  suffices). 
 We now define $S_0=\{ t \in (0, T) : \ab \z'/\z(\s+it) V^{-1/2} \ab \leq \varepsilon \}$, 
where $\varepsilon$ is defined in \eqref{e def}. We have
\[
I_0=\intl_{S_0} \log  \bigg|\zpz(\s+it) V^{-1/2} \bigg| \, dt.
\]

To estimate $I_0$ we partition the set $S_0$ into $\mathcal M_0$ and $\mathcal E_0$, 
where $\mathcal M_0$ is the subset
of $S_0$ consisting of those $t$ such that $D_{\mathbf R}(\s_0+it)$ does not 
contain a zero or pole of the Riemann zeta-function, and  $\mathcal E_0$ contains the remaining $t$.  
We first estimate the integral over $\mathcal E_0$.
\begin{lemma} \label{b0 bd}
Let $1/2+50  \log \log T/ \log T \leq\sigma \leq 2$ and $\mathbf R=\s_0-(\s+1/2)/2$, where $\s_0  \geq 5$
is fixed. Also, let $\mathcal E_0$ denote the set of $0 < t < T$ such that  $D_{\mathbf R}(\s_0+it)$ contains 
a zero or pole of the Riemann zeta-function and such that
$|\z'/\z(\s+it)V^{-1/2}|<\varepsilon$,
where $\varepsilon$ is as defined in \eqref{e def}. Then we have
\[
\intl_{\mathcal E_0} \log  \bigg|\zpz(\s+it) V^{-1/2} \bigg| \, dt \ll \frac{T}{\log^2 T}.
\]
\end{lemma}
\begin{proof}
First, note that
\[
\intl_{\mathcal E_0} \log  \bigg|\zpz(\s+it)V^{-1/2}\bigg| \, dt \ll
\intl_{\mathcal E_0} \log  \bigg|\zpz(\s+it)\bigg| \,dt +O\Big(\meas(\mathcal E_0) \log \log T\Big)
\]
by \eqref{V bd}. Next, by Lemma \ref{simple bd}
\bes
\begin{split}
\intl_{\mathcal E_0} \log   \bigg|\zpz(\s+it) \bigg| \, dt \ll& (\meas(\mathcal E_0))^{1/2}
\bigg(\int_0^T  \bigg(\log   \bigg|\zpz(\s+it) \bigg| \bigg)^2 \, dt \bigg)^{1/2} \\
\ll& (\meas(\mathcal E_0))^{1/2} T^{1/2} \log T.
\end{split}
\ees
Finally, to estimate $\meas(\mathcal E_0)$, observe that if $t \in \mathcal E_0$ is sufficiently large,
then
there is a zero of $\z(s)$, $\rho$, such that $|\rho-\s_0+it| \leq \mathbf R=\s_0-(\s+1/2)/2$.
This implies that $|\beta-\s_0| \leq \s_0 -(\s+1/2)/2$, so $\beta \geq (\s+1/2)/2$. Hence,
for sufficiently large $t \in \mathcal E_0$ we have
\[
t \in \bigcup_{\substack{0< \g \leq T+1 \\ \beta \geq (\s+1/2)/2}} \Big\{t \in (0, T) : 
|t-\g| \leq \mathbf R \Big\}.
\]
As in \eqref{A inf bd} we use
\eqref{zero density} to bound the Lebesgue measure of
the above set. This yields the estimate $\meas(\mathcal E_0) \ll T/\log^6 T$.
Therefore,
\bes
\intl_{\mathcal E_0} \log   \bigg|\zpz(\s+it)V^{-1/2} \bigg| \, dt \ll \frac{T}{\log^2 T}.
\ees
\end{proof}

\begin{lemma} \label{m0 first bd}
Let $1/2+50 \log \log T/ \log T \leq \sigma \leq 2$ 
and $\delta=(\s-1/2)/8$. Also, let $\mathbf R=\s_0-(\s+1/2)/2$, where $\s_0\geq5$ is fixed,
and $R=\mathbf R -\delta$. In addition, define
$\mathcal M_0$ to be the subset of $0<t<T$
such that $D_{\mathbf R}(\s_0+it)$ does not contain a zero or pole of $\z(s)$ and such that
$|\z'/\z(\s+it)V^{-1/2}|<\varepsilon$,
where $\varepsilon$ is defined in \eqref{e def}. Then we have that
\bes 
\begin{split}
\intl_{\mathcal M_0} \log \AB \zpz(\s+it)V^{-1/2}\AB \, dt \ll &
\frac{\log(1/\delta)}{\delta^2} \meas\big(\mathcal M_0\big) \log \log T \\
 &+ \intl_{\mathcal M_0} \sum_{|\vr'-s_0| \leq R-\delta} |\log |\s+it-\vr'|| \, dt,
\end{split}
\ees
where the sum runs over the zeros, $\vr'$, of $\zeta'/\zeta(s)$.
\end{lemma} 
\begin{proof}
Observe that
\be \label{first step 1}
\intl_{\mathcal M_0} \log   \bigg|\zpz(\s+it) V^{-1/2}  \bigg| \, dt=\intl_{\mathcal M_0} \log   \bigg|\zpz(\s+it)  \bigg| \, dt 
+O\Big( \meas(\mathcal M_0) \log \log T \Big).
\ee
To bound the integral on the right  
we use Lemma \ref{decent bd}.  First note that for 
$F(z)=(-2^z/\log 2) \z'/\z(z)$,
$\delta=(\s-1/2)/8$, fixed
$\s_0 \geq 5$, $R=\mathbf{R}-\delta$, there exists
fixed real numbers $r, C >0$ such that
conditions i) and ii) from Section~\ref{complex anal} are satisfied whenever $t \in \mathcal M_0$. 

Let $\vr'$ denote a zero of $F(s)=(-2^s/\log 2) \z'/\z(s)$. (If all the zeros of $\z(s)$ are simple then the
notation $\vr'$ is equivalent to $\rho'$.)  

By Lemma \ref{decent bd} there is an absolute constant $c_2>0$ such that
\bes
\intl_{\mathcal M_0} \log \AB \frac{2^{\s+it}}{\log 2} \zpz(\s+it)\AB \, dt  \geq  -\frac{c_2\log(1/\delta)}{\delta^2} \,\intl_{\mathcal M_0} \log M_t(R) \, dt \\
+\intl_{\mathcal M_0} \sum_{|\vr'-s_0| \leq R-\delta} \log |\s+it-\vr'| \, dt.
\ees
Also, since for $t \in \mathcal M_0$ we have $\z'/\z(\s+it) \ll V^{1/2} \varepsilon \ll \log T$, it follows that
\bes
\intl_{\mathcal M_0} \log \AB \frac{2^{\s+it}}{\log 2} \zpz(\s+it) \AB \, dt \leq  c_3 \meas(\mathcal M_0) \log \log T,
\ees
for some absolute constant $c_3>0$.
 Hence,
\be \label{a0 bd 1}
\begin{split}
\intl_{\mathcal M_0} \log \AB \zpz(\s+it)\AB \, dt \ll &
\frac{\log (1/\delta)}{\delta^2} \,\intl_{\mathcal M_0} \log M_t(R) \, dt 
+  \intl_{\mathcal M_0} \sum_{|\vr'-s_0| \leq R-\delta} |\log |\s+it-\vr'|| \, dt\\
&+  \meas(\mathcal M_0) \log \log T.
\end{split}
\ee

Recall that for $t \in \mathcal M_0$ the disk $D_{\mathbf R}(\s_0+it)$ does
not contain a zero or pole of $\z(s)$. 
Thus, for $t \in \mathcal M_0$ if $z \in D_R(\s_0+it)$ then $z$ is at least
$\delta>1/\log T$ away from a zero or pole of $\z(s)$. By this, \eqref{landau's formula},
and the estimate $N(t+1)-N(t) \ll \log t$ we get that 
for $z \in D_R(\s_0+it)$ where $t \in \mathcal M_0$ 
that $\zeta'/\zeta(z) \ll \log^2 T$. So that whenever $t \in \mathcal M_0$
we have
\be \label{bd for m}
\log M_t(R) \ll \log \log T.
\ee
Thus,
\be \label{a0 bd 2}
\intl_{\mathcal M_0} \log M_t(R) \, dt \ll \meas\(\mathcal M_0\) \log \log T.
\ee
By \eqref{first step 1}, \eqref{a0 bd 1}, and \eqref{a0 bd 2} the lemma follows.
\end{proof}

We now prove
\begin{lemma} \label{m0 second bd}
Let $\mathcal M_0$, $R$, $s_0$, $\s$, and $\delta$ be as in the previous lemma.
We have that
\bes \label{end m0}
\intl_{\mathcal M_0} \sum_{|\vr'-s_0| \leq R-\delta} |\log |\s+it-\vr'|| \, dt \ll 
\Big( \meas\big(\mathcal M_0\big) \Big)^{3/4} T^{1/4} \delta^{-3/2}  \log \log T,
\ees
where $\vr'$ denotes a zero of $\z'/\z(s)$.
\end{lemma}
\begin{proof}
Let $\mathcal A$ denote the subset of $0<t < T$
such that $D_{\mathbf R}(\s_0+it)$ 
does not contain a zero or pole of the 
Riemann zeta-function.
We can partition
$\mathcal A$ into disjoint intervals,
$\mathcal A_1, \ldots, \mathcal A_K$, with
endpoints $a_j$ and $b_j$, so that $0 \leq b_j-a_j \leq 1$. Moreover, by \eqref{zero density}
we can choose the partition so that $K \ll T$. %

Next note that $\mathcal M_0 \subset \mathcal A = \bigcup_{n=1}^K \mathcal A_n$. Applying H\"older's 
inequality yields
\be\label{beginning a0}
\begin{split}
&\intl_{\mathcal M_0}  \sum_{|\vr'-s_0| \leq R-\delta}  |\log |\s+it-\vr'|| \, dt\\
& \qquad \qquad \ll   \Big( \meas\big(\mathcal M_0\big) \Big)^{3/4} 
\bigg(\sum_{n=1}^{K} \int_{a_n}^{b_n} \Big(\sum_{|\vr'-s_0| \leq R-\delta} |\log |\s+it-\vr'|| \Big)^{4} \, dt\bigg)^{1/4}.
\end{split}
\ee
 
 Now consider
\bes
\int_{a_n}^{b_n} \Big( \sum_{|\varrho'-s_0| \leq R-\delta} |\log |\s+it-\varrho'|| \Big)^{4} dt.
\ees
To estimate this we begin by covering the line segment 
$\mathcal S_n=\{\s_0-(R-\delta)+it: a_n \leq t \leq b_n \}$, with disks of the 
form $D_{R}(\s_0+it)$, with $a_n \leq t \leq b_n$.  We first place a disk
of radius $R$ centered at $\s_0+ia_n$ and 
another disk of radius $R$ centered $\s_0+ib_n$. If the
segment is covered, then we are done. If not, we note that both
$D_{R}(\s_0+ia_n)$ and $D_{R}(\s_0+ib_n)$ contain portions of $\mathcal S_n$ 
of length at least  
\[
 \sqrt{R^2-(R-\delta)^2}\geq\sqrt{\delta}.
\]
To cover the remaining portion of $\mathcal S_n$ 
we use the disks $D_{R}(\s_0+i(a_n+m\sqrt{\delta}))$, $m=1,2,\ldots, \lfloor
(b_n-a_n)/\sqrt{\delta} \rfloor$.  Hence, we have 
covered the line segment $\mathcal S_n$ with $\ll |b_n-a_n|/\sqrt{\delta}+1$ disks of the form
$D_{R}(\s_0+it)$, where $a_n \leq t \leq b_n$. 

Let $\mathcal D_n$ be the union of these disks.
Then we have
\bes
\begin{split}
\bigcup_{a_n \leq t\leq b_n} D_{R-\delta}&(\s_0+it)
\subset   \bigg(  D_{R}(\s_0+ia_n) \bigcup D_{R}(\s_0+ib_n) \\
 & \bigcup \{z: \s_0-(R-\delta) \leq \Re(z) \leq \s_0+R-\delta, \,\, a_n \leq \Im(z) \leq b_n \}    \bigg)\\
   \subset &\mathcal{D}_n. 
\end{split}
\ees
It follows that for every $a_n \leq t\leq b_n$ if $|\vr'-s_0| \leq R-\delta$  then $\vr' \in \mathcal D_n$. 
Hence, for $a_n \leq t\leq b_n$, we see that
\[
\sum_{|\vr'-s_0| \leq R-\delta} |\log |\sigma+it-\vr'|| \leq \sum_{\vr' \in \mathcal D_n} |\log |\sigma+it-\vr'||.
\]
Thus,
\be \label{middle e0}
\begin{split}
& \bigg(\sum_{n=1}^{K} \int_{a_n}^{b_n}   \bigg(\sum_{|\vr'-s_0|  \leq R-\delta} |\log |\sigma+it-\vr'||  \bigg)^{4} \,  dt\bigg)^{1/4} \\
& \qquad \qquad \qquad \qquad \qquad\qquad \leq 
\bigg(\sum_{n=1}^{K} \int_{a_n}^{b_n}  \bigg( \sum_{\vr' \in \mathcal D_n} |\log |\sigma+it-\vr'||  \bigg)^{4}\, dt \bigg)^{1/4}.
\end{split}
\ee
By Minkowski's inequality we have uniformly in $n$ that
\be \label{321}
\begin{split}
\bigg(\int_{a_n}^{b_n}  \bigg( \sum_{\vr' \in \mathcal D_n} |\log |\s+it-\vr'||  \bigg)^{4}\, dt \bigg)^{1/4} \ll&
\sum_{\vr' \in \mathcal D_n} \(\int_{a_n}^{b_n}   |\log |\s+it-\vr'||^{4}\, dt \)^{1/4} \\
\ll& \sum_{\vr' \in \mathcal D_n} 1.
\end{split}
\ee

Recall that the set $\mathcal D_n$ is the union of $\ll 1/\sqrt{\delta}$ disks.  Letting
$R^*=(\mathbf R+R)/2$ and applying \eqref{zero bd}, we
see that each one
has $\ll \delta^{-1} \log M_t(R^*) \ll\delta^{-1} \log \log T $ zeros. Here
$M_t(R^*)$ is estimated in the same manner as in \eqref{bd for m}.
Therefore,  the total number of zeros in $\mathcal D_n$ is $\ll \delta^{-3/2} \log \log T$,
uniformly in $n$.
By this, \eqref{beginning a0}, \eqref{middle e0}, and \eqref{321}
we see that
\bes
\intl_{\mathcal M_0} \sum_{|\vr'-s_0| \leq R-\delta} |\log |\s+it-\vr'|| \, dt \ll \Big( \meas\big(\mathcal M_0\big) \Big)^{3/4} T^{1/4} \delta^{-3/2}  \log \log T.
\ees
\end{proof}
 
\begin{lemma} \label{I0 lem} Let $\s=1/2+\psi(T)/\log T$.
For $\log \log T(\log T)^{2/3}\leq \psi(T) \leq \log T/(10 \log \log T)$,
we have
\bes 
I_0 \ll  T \frac{(\log \log T)^3 (\log T)^{3/2}}{\psi(T)^{9/4}}.
\ees

\end{lemma}
\begin{proof}
By Lemma \ref{m0 first bd} and Lemma \ref{m0 second bd} we have that
\bes
\begin{split}
\intl_{\mathcal{M}_0} \, \log \AB \zpz(\s+it) V^{-1/2}\AB \, dt \ll &
 \Big( \meas\big(\mathcal M_0\big) \Big)^{3/4} T^{1/4} \delta^{-3/2}  \log \log T\\
&+\meas\(\mathcal M_0\) \delta^{-2} \log (1/\delta) \log \log T.
\end{split}
\ees
Using \eqref{small measure} we have $\meas\big(\mathcal M_0\big) \ll T \varepsilon^2$. Also, note that for $\psi(T)$ satisfying the hypotheses of the lemma we have $\varepsilon \delta^{-1} \ll 1$. Thus, for $\log \log T (\log T)^{2/3} \leq \psi(T) \leq \log T/(10 \log \log T)$
\bes
\begin{split}
\intl_{\mathcal{M}_0} \, \log \AB \zpz(\s+it) V^{-1/2}\AB \, dt \ll &  T \Big( \frac{\varepsilon}{\delta}\Big)^{3/2} (\log \log T)^2 \\
\ll & T \frac{(\log \log T)^3 (\log T)^{3/2}}{\psi(T)^{9/4}}.  
\end{split}
\ees
The lemma follows from this and Lemma \ref{b0 bd}.
\end{proof}
 
 \subsection{Estimation of $I_M$ and the Proof of Theorem~\ref{main est}}
 
 We first prove
 \begin{lemma}\label{I_M lemma} Let $(\log T)^{-1/3} \leq 2\s-1 \leq 1/(10 \log \log T)$ and let 
 \[
\bOmega= e^{-10} \min\big(V^{3/2}, (\psi(T)/ \log \psi(T))^{1/2}\big).  
 \]
 Then
 \[
 I_M= (\log 2-\g)\frac{T}{2}+O \bigg(T \frac{(\log \log T)^2}{\bOmega} \bigg),
 \]
 where $\g=0.57721566...$ is Euler's constant.
 \end{lemma}

   \begin{proof} 
We begin by setting 
 \[
 S_M= \bigg\{t \in (0, T) : \varepsilon 
\leq \AB \zpz(\s+it) V^{-1/2} \AB \leq \mathcal N  \bigg\},
 \]
 where $\varepsilon$ is defined in \eqref{e def} 
 and $\mathcal N$ is defined in \eqref{def N}.
By  definition,  
\[
 I_M= \intl_{S_M} \log \AB \zpz(\s+it) V^{-1/2} \AB dt.
 \]
Let
\[
\Psi_T(r)=\frac1T  \, \, \meas\bigg\{t \in [0, T] : 
\AB \frac{\z'}{\z}(\s+it)\AB\leq \sqrt{V}r  \bigg\}
\]
and
\[
\Phi_T(r)=\frac1T  \, \, \meas\bigg\{t \in [0, T] : 
\AB \frac{\z'}{\z}(\s+it)\AB > \sqrt{V}r  \bigg\}.
\]
Note that $\Psi_T(r)$ is increasing so it is of bounded variation.
By the change of variables formula (see ~\cite{Dur} Theorem 1.6.9)
we have
\[
\frac1T \intl_{S_M} \log \AB \zpz(\s+it) V^{-1/2} \AB \, dt=\int_{[\varepsilon, \mathcal N]} \log r \, \, d\Psi_T(r),
\]
here the integral is a Lebesgue-Stieltjes integral.

To estimate the integral we integrate by parts (see Section 6.4 of ~\cite{Burk}) to obtain
\bes
\frac1T \cdot I_M=\Psi_T(\mathcal N^+)\log \mathcal N-\Psi_T(\varepsilon^-)\log \varepsilon -\int_{[\epsilon,\mathcal N]}  \Psi_T(r) \frac{dr}{r},
\ees
where $\Psi_T(\mathcal N^+)$ is the limit of $\Psi_T(r)$ as $r$ approaches $\mathcal N$ from the right
and $\Psi_T(\varepsilon^-)$ is the limit of $\Psi_T(r)$ as $r$ approaches $\varepsilon$ from the left.
Note that since $\Psi_T(r)$ is continuous from the right $\Psi_T(\mathcal N^+)=\Psi_T(\mathcal N)$
and by \eqref{small measure} $\Psi_T(\varepsilon^-) \leq \Psi_T(\varepsilon) \ll \varepsilon^2$.
Also since $\Psi_T(r)$ is bounded and increasing it 
has at most countably many discontinuities. Hence, $\Psi_T(r)/r$
is Riemann integrable on $[\varepsilon, \mathcal N]$.
By these observations we have
\be \label{int by parts formula}
\frac1T \cdot I_M=\Psi_T(\mathcal N)\log \mathcal N -\int_{\epsilon}^{\mathcal N}  \Psi_T(r) \frac{dr}{r}+O\big(\varepsilon^2 \log (1/\varepsilon) \big),
\ee
where the integral on the right-hand side is a Riemann integral.

Next, note that $\Psi_T(r)=1-\Phi_T(r)$
and $\Phi_T(r)$ is a decreasing function in $r$. So by \eqref{main measure}
we have for $r \geq \sqrt{10 \log \log T}$  
that
\bes
\begin{split}
\Phi_T(r) \leq \Phi_T(\sqrt{10 \log \log T})
=& \exp(-5\log \log T)+O( (\log \log T)/\bOmega) \\
\ll& (\log \log T)/\bOmega.
\end{split}
\ees
Thus, for $r \geq \sqrt{10 \log \log T}$ 
\be \label{bound for psi}
\Psi_T(r)=1+O\bigg(\frac{\log \log T}{\bOmega} \bigg).
\ee
Note $\mathcal N > \sqrt{10 \log \log T}$,
so applying this estimate in \eqref{int by parts formula}
we see that
\be \label{after int by parts}
\frac1T \cdot I_M=\log \mathcal N-\int_{\varepsilon}^{\mathcal N}  \Psi_T(r) \frac{dr}{r}
+O\bigg(\frac{(\log \log T)^2}{\bOmega}\bigg)+O(\log (1/\varepsilon)\varepsilon^2).
\ee

Now write
\[
\log \mathcal N=\int_{1}^{\mathcal N} \, \frac{dr}{r}
\]
and note that
\[
\int_{\varepsilon}^{\mathcal N}  \Psi_T(r) \frac{dr}{r}
=\int_{1}^{\mathcal N} (1- \Phi_T(r)) \frac{dr}{r}
+\int_{\varepsilon}^{1}  \Psi_T(r) \frac{dr}{r}.
\]Using these observations in \eqref{after int by parts} we have
\be \label{all together}
\frac1T \cdot I_M=\int_1^{\mathcal N} \Phi_T(r) \, \frac{dr}{r}-\int_{\varepsilon}^{1} \Psi_T(r) \, \frac{dr}{r}+O \bigg(\frac{(\log \log T)^2}{\bOmega}+\log (1/\varepsilon)\varepsilon^2  \bigg).
\ee

By \eqref{main measure} and \eqref{bound for psi} we have
\[
\int_1^{\mathcal N} \, \Phi_T(r) \frac{dr}{r}=\int_1^{\sqrt{10 \log \log T}} \, \bigg(e^{-r^2/2}+O\bigg(\frac{r^2}{\bOmega}\bigg)\bigg) \frac{dr}{r}
+\int_{\sqrt{10 \log \log T}}^{\mathcal N} \, O\bigg(\frac{\log \log T}{\bOmega}\bigg) \frac{dr}{r}.
\]
Hence,
\[
\int_1^{\mathcal N} \, \Phi_T(r) \frac{dr}{r}=\int_1^{\infty} e^{-r^2/2} \frac{dr}{r}+O \bigg(\frac{(\log \log T)^2}{\bOmega}  \bigg).
\]
By \eqref{main measure} and \eqref{small measure}, it also follows that
\[
\int_{\varepsilon}^{1} \, \Psi_T(r) \frac{dr}{r}=\int_0^{1} \,(1- e^{-r^2/2}) \frac{dr}{r}+O(1/\bOmega). 
\]

Finally, integration by parts yields
\[
\int_1^{\infty} \, e^{-r^2/2} \frac{dr}{r}-\int_0^{1} \,(1- e^{-r^2/2}) \frac{dr}{r}=\int_0^{\infty} r \log r \, e^{-r^2/2} dr =\tfrac12(\log 2-\g),
\]
where $\g=0.57721566...$ denotes Euler's constant,
and the last calculation was done using Mathematica. 
To complete the  proof of the lemma, we insert these estimates into 
\eqref{all together} and note that $\log (1/\varepsilon) \varepsilon^{2} \ll
(\log \log T)^2/\bOmega$ for $\s$ satisfying the hypotheses of the lemma.
\end{proof}

\begin{proof}[Proof of Theorem \ref{main est}]
Recall that
\[
\int_0^T \log \AB \zpz(\s+it)V^{-1/2} \AB \, dt=
I_M+I_0+I_{\infty}.
\]
By Lemmas ~\ref{I infty bd}, \ref{I0 lem},  and \ref{I_M lemma} we have
\bes
\begin{split}
 \int_0^T \log \AB \zpz(\s+it)V^{-1/2} \AB \, dt
=&\frac{T}{2}(\log 2-\gamma)+O\bigg( T \frac{(\log \log T)^2}{\bOmega}\bigg)\\
&+O\bigg(T \frac{(\log \log T)^3 (\log T)^{3/2}}{\psi(T)^{9/4}}
\bigg)+O(T/\log^2 T).
\end{split}
\ees
Next, note that
  \bes
  \begin{split}
 \int_0^T \log \AB \zpz(\s+it)V^{-1/2} \AB \, dt=&\int_0^T \log \AB \zpz(\s+it) \AB \, dt\\
 &-T \log \Big(\frac{1}{2\s-1} \Big)+\log 2\frac{T}{2}+O(T/V).
 \end{split}
 \ees
The result follows after collecting estimates and then combining and simplifying the error terms.
\end{proof}
 
 \section{The Proof of Theorem \ref{thm 1}} \label{section 4}
 
Applying Littlewood's Lemma (see Titchmarsh ~\cite{T} equation 9.9.2) to $(-2^s/\log 2)\z'(s)$ 
and taking imaginary parts, we have
\be \notag
\begin{split}
2\pi \sum_{\substack{1< \g' \leq T \\  \beta'>\s}} (\beta'-\s)&=
 \int_1^T \! \log \AB \frac{-2^{\s+it}}{\log 2}\z'(\s+it) \AB \, dt
-\int_1^T \! \log \AB \frac{-2^{10+it}}{\log 2}\z'(10+it) \AB \, dt \\
&  +\int_{\s}^{10}  \! \arg \Big( \frac{-2^{\alpha+iT}}{\log 2}\z'(\alpha+iT) \Big) \, d\alpha- 
\int_{\s}^{10} \arg \Big( \frac{-2^{\alpha+i}}{\log 2}\z'(\alpha+i) \Big)  \, d\alpha.
\end{split}
\ee
Clearly, the last integral is $\ll 1$. To estimate the third integral we note that
it follows from
Berndt ~\cite{Be} that $\arg ((-2^{\s+iT}/\log 2)\z'(\s+iT)) \ll \log T$. By
a standard argument using contour integration 
we can show that the second integral is $O(1)$. Thus, we get that
\be \label{zeros of zeta prime}
2\pi \sum_{\substack{ \beta'>\s \\ 0 < \g' \leq T}} (\beta'-\s)=  \int_0^T \! \log \ab \z'(\s+it) \ab \, dt+
T \log \frac{2^{\s}}{\log 2}+O(\log T).
\ee

From Theorem 9.15 of Titchmarsh ~\cite{T} it follows that
\be \label{zeros of zeta}
2\pi \sum_{\substack{ 0 < \g \leq T \\ \beta >\s}} (\beta-\s)=\int_0^T \! \log \ab \z(\s+it) \ab \, dt+O(\log T).
\ee
Now, by \eqref{zero density} and our assumption 
that $\s>1/2+(\log T)^{-1/3}$,   the sum on the left-hand side of
the above equation is
\[
\ll \sum_{\substack{ 0 < \g \leq T \\ \beta>\s}} 1 \ll T^{1-(\s-1/2)/4} \log T \ll T/\log^{2} T.
\]
Therefore, differencing equations \eqref{zeros of zeta prime}
and \eqref{zeros of zeta} gives
\[
2\pi \sum_{\substack{ 0< \g' \leq T \\  \beta'>\s}} (\beta'-\s)=\int_0^T \! \log \AB \zpz(\s+it) \AB \, dt
+T \log \frac{2^{\s}}{\log 2}+O(T/\log^2 T).
\]
Theorem \ref{thm 1} now follows from Theorem \ref{main est}.


\section{The Proof of Theorem \ref{thm 2.1}} \label{section 5}

By Theorem \ref{thm 1}  we have 
\[
\sum_{\substack{0< \g' \leq T \\  \beta'>\s}} (\beta'-\s)=\frac{T}{2\pi}\log \Big( \frac{1}{2\s-1}\Big)
+\Big(\log \frac{2^{\s}}{\log 2}-\frac{\g}{2}\Big)\frac{T}{2\pi} + O(E((2\s-1)\log T)).
\]
Throughout this section we write this as 
\[
L(\s)=R(\s).
\]
Before proving Theorem \ref{thm 2.1}, we prove

\begin{lemma} \label{deriv lemma} Let $ \log \log T/(\log T)^{1/3} \leq 2\s-1 \leq (20 \log \log T)^{-1}$
and $0<\D<(2\s-1)/4$. Then  
\bes
\begin{split}
\frac{R(\s)-R(\s+\D)}{\D} = &
\frac{T}{2 \pi(\s-\tfrac12)}- \frac{T}{2\pi}\log 2\\
&+O\( T \frac{\Delta}{(2\s-1)^2}   + TE((2\s-1)\log T)/\D\).
\end{split}
\ees
\end{lemma}
\begin{proof}
First we note that
by definition
\bes 
\begin{split}
\frac{R(\s)-R(\s+\D)}{\D}=& \frac{T}{2 \pi  \D } \log \(\frac{2(\s+\D)-1}{2\s-1}\)
-\frac{T}{2\pi} \log 2\\
&+ O(T(E((2\s-1)\log T)+E((2(\s+\D)-1)\log T))/\D).
\end{split}
\ees  
Since $\D <(2\s-1)/4$,
we have that $E((2\s-1)\log T)\approx E((2(\s + \D)-1)\log T)$.  
Next, note that we may approximate the logarithm in the above equation by the first two terms of its Taylor series. 
Thus, the right-hand side equals
\bes 
\begin{split}
\frac{T}{2 \pi \Delta} \cdot \frac{2\Delta}{2\s-1}+O\( T \frac{\Delta}{(2\s-1)^2} \) -\frac{T}{2\pi} \log 2+ O(TE((2\s-1)\log T)/\D),
\end{split}
\ees
as desired.
\end{proof}


\begin{proof}[Proof of Theorem \ref{thm 2.1}]
Again, observe that
\[
L(\sigma)=\int_{\sigma}^{\infty} \sum_{\substack{0<
\gamma' \leq T \\ \beta'>u}} 1 \, du.
\]
Thus,
\[
\frac{R(\sigma)-R(\sigma+\D)}{\D}=\frac{L(\sigma)-L(\sigma+\Delta)}{\Delta}=\frac{1}{\Delta}
\int_{\sigma}^{\sigma+\Delta} \sum_{\substack{0<
\gamma' \leq T \\ \beta'>u}} 1 \, du.
\]
So that
\[
\sum_{\substack{0<
\gamma' \leq T \\ \beta'>\sigma+\Delta}} 1 \leq \frac{R(\sigma)-R(\sigma+\D)}{\D} \leq \sum_{\substack{0<
\gamma' \leq T \\ \beta'>\sigma}} 1.
\]
Similarly,
\[
\sum_{\substack{0<
\gamma' \leq T \\ \beta'>\sigma}} 1 \leq \frac{R(\sigma-\Delta)-R(\sigma)}{\D} \leq \sum_{\substack{0<
\gamma' \leq T \\ \beta'>\sigma-\Delta}} 1.
\]
Thus,
\be \label{above below}
\frac{R(\sigma)-R(\sigma+\D)}{\D} \leq \sum_{\substack{0<
\gamma' \leq T \\ \beta'>\sigma}} 1 \leq
\frac{R(\sigma-\Delta)-R(\sigma)}{\D}.
\ee

By Lemma \ref{deriv lemma}
we have for $0<\D<(2\sigma-1)/4$ and $ 2\log \log T/(\log T)^{1/3} \leq 2\sigma-1 \leq (20 \log \log T)^{-1}$
that
\bes
\begin{split}
\frac{R(\sigma)-R(\sigma+\D)}{\D} = &
\frac{T}{2 \pi(\sigma-\tfrac12)}- \frac{T}{2\pi}\log 2\\
&+O\( T \frac{\Delta}{(2\sigma-1)^2}   + TE((2\sigma-1)\log T)/\D\),
\end{split}
\ees
and,
\bes
\begin{split}
\frac{R(\sigma-\Delta)-R(\sigma)}{\D} = &
\frac{T}{2 \pi(\sigma-\tfrac12)}- \frac{T}{2\pi}\log 2\\
&+O\( T \frac{\Delta}{(2\sigma-1)^2}   + TE((2(\sigma-\D)-1)\log T)/\D\).
\end{split}
\ees
Since $0< \D \leq (2\s-1)/4$ we have $E((2\s-1)\log T) \approx E((2(\sigma-\D)-1)\log T)$. Applying
these estimates in \eqref{above below} yields
\[
\sum_{\substack{0<
\gamma' \leq T \\ \beta'>\sigma}} 1=
\frac{T}{2 \pi(\sigma-\tfrac12)}- \frac{T}{2\pi}\log 2
+O\( T \frac{\Delta}{(2\sigma-1)^2}   + TE((2\sigma-1)\log T)/\D\).
\]
To balance the error terms take $\Delta=(2\s-1)\sqrt{E(\psi(T))}$. Noting that the term
$-T/(2\pi) \log 2$ is absorbed by the error term here, the theorem follows.
\end{proof}

\subsection*{Acknowledgments}
This article is from the author's PhD thesis, which was supervised by Prof. Steven Gonek. I would like 
to thank Prof. Gonek for his guidance and support.

\end{document}